\documentclass[11pt]{amsart}

\usepackage{amsfonts,amsmath,amssymb,amsthm}
\usepackage{mathrsfs}
\usepackage{color}

\newtheorem{thm}{Theorem}[section]
\newtheorem{lem}[thm]{Lemma}
\newtheorem{prp}[thm]{Proposition}
\newtheorem{cor}[thm]{Corollary}
\theoremstyle{definition}
\newtheorem{dfn}[thm]{Definition}
\theoremstyle{remark}
\newtheorem{rem}[thm]{Remark}
\newtheorem{ex}[thm]{Example}
\newtheorem{notation}[thm]{Notation}

\numberwithin{equation}{section}

 \DeclareMathOperator{\grad}{grad}

\DeclareMathOperator{\trace}{trace}
\DeclareMathOperator{\curl}{curl}
\begin{document}
\thispagestyle{empty}

\newpage

\title[Killing Vector Fields of Standard Static Space-times]
{Killing Vector Fields of Standard Static Space-times}
\author{Fernando Dobarro}
\address[F. Dobarro]{Dipartimento di Matematica e Informatica,
Universit\`{a} degli Studi di Trieste, Via Valerio 12/B, I-34127
Trieste, Italy} \email{dobarro@dmi.units.it}
\author{B\"{u}lent \"{U}nal}
\address[B. \"{U}nal]{Department of Mathematics, Bilkent University,
         Bilkent, 06800 Ankara, Turkey}
\email{bulentunal@mail.com}
%
%
%
\date{\today}

\begin{abstract}
We consider Killing vector fields on standard static space-times and
obtain equations for a vector field on a standard static space-time
to be Killing. We also provide a characterization of Killing vector
fields on standard static space-times with compact Riemannian parts.
\end{abstract}

\maketitle

\tableofcontents

\newpage

\renewcommand{\thepage}{\arabic{page}}
\setcounter{page}{1}

\section{Introduction}

Warped product manifolds were introduced in general relativity as
a method to find general solutions to Einstein's field equations
\cite{BEE,ON}. Two important examples include generalized
Robertson-Walker space-times and standard static space-times.  The
former are obviously a generalization of Robertson-Walker
space-times and the latter a generalization of the Einstein static
universe. In this paper, we focus on Killing vector fields for standard
static space-times.

We recall that a warped product can be defined as follows
\cite{BEE,ON}. Let $(B,g_B)$ and $(F,g_F)$ be pseudo-Riemannian
manifolds and also let $b \colon B \to (0,\infty)$ be a smooth
function. Then the (singly) warped product, $B \times_b F$ is the
product manifold $B \times F$ furnished with the metric tensor
$g=g_B \oplus b^{2}g_F$ defined by $$ g=\pi^{\ast}(g_B) \oplus (b
\circ \pi)^2 \sigma^{\ast}(g_F)$$ where $\pi \colon B \times F \to
B$ and $\sigma \colon B \times F \to F$ are the usual projection
maps and ${}^\ast$~denotes the pull-back operator on tensors. A
standard static space-time can be considered as a Lorentzian
warped product where the warping function is defined on a
Riemannian manifold called the base and acting on the negative
definite metric on an open interval of real numbers, called the
fiber. More precisely, a standard static space-time $(t_1,t_2) _f
\times F$ is a Lorentzian warped product furnished with the metric
$g=-f^2{\rm d}t^2 \oplus g_F,$ where $(F,g_F)$ is a Riemannian
manifold, $f \colon F \to (0,\infty)$ is smooth and $-\infty \leq
t_1 < t_2 \leq \infty.$ In \cite{ON}, it was shown that any static
space-time is locally isometric to a standard static space-time.

\noindent Standard static space-times have been previously studied
by many authors. Kobayashi and Obata \cite{KO} stated the geodesic
equation for this class of space-times and the causal structure
and geodesic completeness was considered in \cite{AD}, where
sufficient conditions on the warping function for nonspacelike
geodesic completeness of the standard static space-time was
obtained (see also \cite{RASM}). In \cite{AD1}, conditions are
found which guarantee that standard static space-times either
satisfy or else fail to satisfy certain curvature conditions from
general relativity. In \cite{ADt}, D. Allison considered the
global hyperbolicity of standard static space-times and obtained
sufficient conditions. The problems of geodesic completeness and
global hyperbolicity of standard static space-times have been also
studied by M. S\'{a}nchez in \cite{MS05} with $I=\mathbb{R}$ where
$(F,g_F)$ behaves at most quadratically (at infinity). The
existence of geodesics in standard static space-times have been
studied by several authors. S\'{a}nchez \cite{mS2} gives a good
overview of geodesic connectedness in semi-Riemannian manifolds,
including a discussion for standard static space-times. The
geodesic structure of standard static space-times has been studied
in \cite{GES} and conditions are found which imply nonreturning
and pseudoconvex geodesic systems. As a consequence, it is shown
that if the complete Riemannian factor manifold $F$ satisfies the
nonreturning property and has a pseudoconvex geodesic system and
if the warping function $f \colon F \to (0,\infty)$ is bounded
from above then the standard static space-time $(a,b)_f\times F $
is geodesically connected. In \cite{SSS}, some conditions for the
Riemannian factor and the warping function of a standard static
space-time are obtained in order to guarantee that no nontrivial
warping function on the Riemannian factor can make the standard
static space-time Einstein. In a recent note
\cite{Sanchez-Senovilla2007}, the authors discussed conditions for
static Killing vector fields to be standard and then they obtained
an interesting uniqueness result when the so called natural space
(in the case of a standard static space-time, this is the
Riemannian part) is compact.

\noindent Two of the most famous examples of standard static
space-times are the Minkowski space-time and the Einstein static
universe \cite{BEE,dewitt 2003,HE} which is $\mathbb R \times
\mathbb S^3$ equipped with the metric
$$g=-{\rm d}t^2+({\rm d}r^2+\sin^2 r {\rm d}\theta^2+
\sin^2 r \sin^2 \theta {\rm d} \phi^2)$$ where $\mathbb S^3$ is
the usual 3-dimensional Euclidean sphere and the warping function
$f \equiv 1$ (\textit{see Remark \ref{rem:sharipov2007}}).
Another well-known example is the universal covering space of
anti-de Sitter space-time, a standard static space-time of the
form $\mathbb R _f\times \mathbb H^3$ where $\mathbb H^3$ is the
3-dimensional hyperbolic space with constant negative sectional
curvature and the warping function $f \colon \mathbb H^3 \to
(0,\infty)$ defined as $f(r,\theta,\phi)=\cosh r$ \cite{BEE,HE}.
Finally, we can also mention the Exterior Schwarzschild space-time
\cite{BEE,HE}, a standard static space-time of the form $\mathbb R
_f\times (2m,\infty) \times \mathbb S^2,$ where $\mathbb S^2$ is
the 2-dimensional Euclidean sphere, the warping function $f \colon
(2m,\infty) \times \mathbb S^2 \to (0,\infty)$ is given by
$f(r,\theta,\phi)=\sqrt{1-2m/r},$ $r>2m$ and the line element on
$(2m,\infty) \times \mathbb S^2$ is
$${\rm d}s^2=\left(1-\frac{2m}{r} \right)^{-1} {\rm d}r^2+r^2({\rm
d} \theta^2+\sin^2 \theta {\rm d} \phi^2).$$

In this article, we deal with questions of existence and
characterization of Killing vector fields in standard static space
times\footnote{We would like to inform the reader that some of the
results provided in this article were previously announced in the
arXiv preprint service as \cite{dobarro-unal06}.}.

\noindent The problem of the existence of Killing vector fields in
semi-Riemannian manifolds has been analyzed by many authors
(physicists and mathematicians) with different points of view and by
using several techniques. One of the recent articles of S\'{a}nchez
(i.e, \cite{mS1997}) is devoted to provide a review about these
questions in the framework of Lorentzian geometry. Another
interesting results for Riemannian warped products can be found in
\cite{B-O} and for $4-$dimensional warped space-times in
\cite{Carot-da Costa 93}.

\noindent In \cite{mSK} S\'{a}nchez studied the structure of
Killing vector fields on a generalized Robertson-Walker
space-time. He obtained necessary and sufficient conditions for a
vector field to be Killing on generalized Robertson-Walker
space-times and gave a characterization of them as well as an
explicit list for the globally hyperbolic case.

After this brief explanations of some of the major works in the
geometry of warped products, especially in the geometry of standard
static space-times, we will provide an outline of the paper below.

In {\it Section 2}, we study Killing vector fields of standard
static space-times. Firstly we show necessary and sufficient
conditions for a vector field of the form $h \partial_t + V$ to be a
conformal Killing (see {\it Proposition \ref{con-1}}).

\noindent Then adapting the techniques of S\'{a}nchez in \cite{mSK} to
standard static space-times $M=I_f \times F$, we give the
structure of a generic Killing vector field on $M$ in the central
result of this section (see {\it Theorem \ref{thm:Killing ssst}}).
Essentially, we reduce the problem to the study of a parametric
system of partial differential equations (involving the Hessian) on
the Riemannian part $(F, g_F)$.

\noindent By studying on the latter system (see \eqref{eq:pb
Killing 1} and \eqref{eq:3eqnst 4}) and applying the well known
results about the solutions $(\nu, u)$ of a weighted elliptic
problem, with $w \in C^\infty_{>0}$,
\begin{equation*}\label{eq1:}
    -\Delta_{g_F} u = \nu w
    u \, \textrm{ on } (F,g_F).
\end{equation*}
on a compact Riemannian manifold without boundary $(F,g_F)$,
we characterize the Killing vector fields on a standard
static space-time with \textit{compact} Riemannian part in
{\it Theorem \ref{thm:killing compact fiber}}. More explicitly,
if $(F,g_F)$ is a compact Riemannian manifold, then the set of
Killing vector fields on a standard static space time $I _f\times F$ is
$$\{a
\partial_t + \tilde{K}| \, a \in \mathbb{R}, \tilde{K} \textrm{ is a
Killing vector field on } (F,g_F) \textrm{ and } \tilde{K}(f)=0
\}.$$

In Remark, \ref{rem:sharipov2007} we show that a relation between
the result mentioned above and an article (see \cite{Sharipov2007})
of R. A. Sharipov about Killing vector fields of a closed homogeneous
and isotropic universe.

In {\it Section 3}, as an application of our results in {\it Section 2}
we consider the question of the existence of non-rotating Killing
vector fields on a standard static space-time where the Riemannian
part is compact and simply connected.

In \textit{Appendix}, we provide an equivalent expression of a
general Killing vector field on a standard static space-time and
an application of our results to the well known case of the
Minkowski space-time.

\section{Killing Vector Fields}

We begin by stating the formal definition of a standard static
space-time and then recall some elementary definitions and facts
(see Section 3 of \cite{mSK}). But first, we want to emphasize
some important notational issues. Throughout the paper $I$ will
denote an open connected interval of the form $I=(t_1,t_2)$ in
$\mathbb R,$ where $-\infty \leq t_1 < t_2 \leq \infty.$ Moreover,
any underlying manifold is assumed to be connected. Finally, on an
arbitrary differentiable manifold $N$, $C^{\infty}_{>0}(N)$
denotes the set of all strictly positive $C^{\infty}$ functions
defined on $N$ and $\mathfrak X(N)$ will denote the
$C^{\infty}(N)-$module of smooth vector fields on $N$.

\begin{dfn} \label{dfna} Let $(F,g_F)$ be an $s$-dimensional
Riemannian manifold and $f \colon F \to (0,\infty)$ be a smooth
function. Then $n(=s+1)$-dimensional product manifold $I \times
F$ furnished with the metric tensor $g=-f^2{\rm d}t^2 \oplus g_F$
is called a standard static space-time and is denoted by
$I _f \! \times F$, where ${\rm d}t^2$ is the Euclidean metric
tensor on $I.$
\end{dfn}

We will now define Killing and conformal-Killing vector fields on
an arbitrary pseudo-Riemannian manifold. Let $(\overline M,
\overline g)$ be an ${\overline n}$-dimensional pseudo-Riemannian
manifold and $X \in \mathfrak X(\overline M)$ be a vector field on
$\overline M.$ Then

\begin{itemize}
\item $X$ is said to be Killing if ${\rm L}_X \overline g = 0$,

\item $X$ is said to be conformal-Killing if there exists a smooth
function $\sigma \colon \overline M \to \mathbb R$ such that ${\rm
L}_X \overline g = 2 \sigma \overline g$,
\end{itemize}
where ${\rm L}_X$ denotes the Lie derivative with respect to $X$.
Moreover, for any $Y$ and $Z$ in $\mathfrak X(\overline M),$ we
have the following identity (see \cite[p.162 and p.61]{ON})

\begin{equation}\label{eq:Lie deriv 1}
    {\rm L}_X \overline g(Y,Z) = \overline g(\nabla_Y X,Z) + \overline
g(Y,\nabla_Z X).
\end{equation}


\begin{rem} \label{rem:Lie deriv}
On $(I,g_I=\pm {\rm d}t^2)$ any vector field is conformal Killing.
Indeed, if $X$ is a vector field on $(I,g_I)$, then $X$ can be
expressed as $X = h \partial_t$ for some smooth function
$h \in \mathcal{C}^\infty(I)$. Hence, ${\rm L}_X g_I=2 \sigma g_I$
with $\sigma=h^\prime$.
\end{rem}

We will now state a simple result which will be useful in our
study (see \cite{mSK}, \cite{DWP} and page 126 of \cite{PHD}).

\begin{prp} \label{kill-m} Let $M=I_f \times F$ be a standard
static space-time with the metric $g=f^2 g_I \oplus g_F$, where
$g_I=-{\rm d}t^2$. Suppose that $X,Y,Z \in \mathfrak X(I)$ and
$V,W,U \in \mathfrak X(F).$ Then
\begin{eqnarray*} {\rm L}_{X+V}g(Y+W,Z+U) & = & f^2 {\rm L}^I_X
g_I(Y,Z) + 2f V(f)g_I(Y,Z) \\
& + & {\rm L}^F_V g_F(W,U)
\end{eqnarray*}
\end{prp}

Note that if $h \colon I \to \mathbb R$ is smooth and $Y,Z \in
\mathfrak X(I),$ then

\begin{equation}\label{eq:Lie deriv 2}
{\rm L}_{h \partial_t} g_I(Y,Z) = Y(h)g_I(Z,\partial_t) + Z(h) g_I
(Y, \partial_t).
\end{equation}

By combining the previous statements we can prove the following:

\begin{prp} \label{con-1} Let $M=I_f \times F$ be a standard
static space-time with the metric $g=-f^2{\rm d}t^2 \oplus g_F.$
Suppose that $h \colon I \to \mathbb R$ is smooth and $V$ is a
vector field on $F.$ Then $h
\partial_t + V$ is a conformal-Killing vector
field on $M$ with $\sigma \in C^{\infty}(M)$ if and only if the
followings are satisfied:
\begin{enumerate}
\item $V$ is conformal-Killing on $F$ with $\sigma \in
C^{\infty}(F),$

\item $h$ is affine, i.e, there exist real numbers $\mu$ and $\nu$
such that $h(t)=\mu t + \nu$ for any $t \in I,$

\item $V(f) = (\sigma-\mu) f.$
\end{enumerate}
\end{prp}

\begin{proof}(1) follows from {\it Proposition \ref{kill-m}}
by taking $Y=Z=0$ and a separation of variables argument. On the
other hand, from {\it Proposition \ref{kill-m}} with $W=U=0$,
\eqref{eq:Lie deriv 2} and Remark \ref{rem:Lie deriv}, we have
$(\sigma - h^\prime)f=V(f)$. Hence, again by separation of
variables, $h^\prime$ is constant and then (2) is obtained. Thus,
(3) is clear.

By using similar computations described as above, the converse
turns out to be a consequence of the decomposition of any vector
field on $M,$ i.e., as a sum of its horizontal and vertical parts.
\end{proof}

\begin{cor}\label{cor:simple Killing}
Let $M=I _f\times F$ be a standard static space-time with the
metric $g=-f^2{\rm d}t^2 \oplus g_F.$ Suppose that $h \colon I \to
\mathbb R$ is smooth and $V$ is a vector field on $F$. Then $h
\partial_t + V$ is a Killing vector field on $M$ if and only if
the followings are satisfied:
\begin{enumerate}
\item $V$ is Killing on $F$,

\item $h$ is affine, i.e, there exist real numbers $\mu$ and $\nu$
such that $h(t)=\mu t + \nu$ for any $t \in I$,

\item $V(f) = - \mu f$.
\end{enumerate}
\end{cor}

\begin{proof}
It is sufficient to apply Proposition \ref{con-1} with $\sigma
\equiv 0$.
\end{proof}

\begin{notation}\label{notation: It0} Let $h$ be a
given a continuous function defined on a real interval $I$. If
there exists a point $t_0 \in I$ such that $h(t_0)\neq 0$, then
$I_{t_0}$ denotes the connected component of $\{t \in I: h(t)\neq
0\}$ such that $t_0 \in I_{t_0}$.
\end{notation}

In what follows, we will make use part of the arguments given in
\cite{mSK} (see also \cite{Carot-da Costa 93}) about the structure
of Killing and conformal-Killing vector fields in warped products.
In \cite{mSK} by applying such arguments, M. S\'{a}nchez obtains
full characterizations of the Killing and conformal-Killing vector
fields in generalized Robertson-Walker space-times. In order to be
more explanatory, we begin by adapting his procedure to our set of
frame.

\bigskip

Let $(B, g_B)$ and $(F, g_F)$ be two semi-Riemannian manifolds
with dimensions $r$ and $s > 0$ respectively, and also let $f \in
C_{>0}^{\infty}(F)$ be. Consider the warped metric $g= f^2g_B +
g_F$ on $B \times F$. Given a vector field $Z$ on $B \times F$, we
will write $Z = Z_B + Z_F$ with $Z_B = ({\pi_B}_*(Z), 0)$ and $Z_F
= (0,{\pi_F}_*(Z))$, the projections onto the natural foliations
($B_q = B \times \{q\}$, $q \in F$ and $F_p=\{p\} \times F, p\in
B$). Any covariant or contravariant tensor field $\omega$ on one
of the factors ($B$ or $F$) induces naturally a tensor field on $B
\times F$ (i.e. the lift), which either will be denoted by the
same symbol $\omega$, or else (when necessary) will be
distinguished by putting a bar on it, i.e, $\overline{\omega}$.

\bigskip

\begin{prp}\label{prp:sanchez99}
(see Proposition 3.6 in \cite{mSK}) If $K$ is a Killing vector
field on $M=B_f \times F$, then $K_B$ is a conformal Killing vector
field on $B_q$ for any $q \in F$ and $K_F$ is a Killing vector field
on $F_p$ for any $p \in B$.
\end{prp}

\bigskip

Suppose that $\{C_{\overline a} \in \mathfrak X(B) | \, \overline
a = 1,\cdots, \overline r \}$ is a basis for the set of all
conformal-Killing vector fields on $B$ and $\{K_{\overline b} \in
\mathfrak X(F) | \, \overline b = 1,\cdots, \overline s \}$ is a
basis for the set of all Killing vector fields on $F.$

According to \cite{mSK} (see also \cite[Sections 7 and 8]{B-O}),
Killing vector fields on a warped product of the form $M=B_f \times F$
with the metric $g=f^2 g_B \oplus g_F$ can be given as

\begin{equation} \label{form}
K=\psi^{\overline a} C_{\overline a} + \phi^{\overline b}
K_{\overline b},
\end{equation}
where $\phi^{\overline b} \in \mathcal C^\infty(B)$ and
$\psi^{\overline a} \in \mathcal C^\infty(F).$

Define $K_B=\psi^{\overline a} C_{\overline a}$ and
$K_F=\phi^{\overline b} K_{\overline b}$. Moreover,
$\hat{K}_{\overline b}=g_F(K_{\overline b},\cdot)$ and
$\hat{C}_{\overline a}=g_B(C_{\overline a},\cdot).$

Then Proposition 3.8 of \cite{mSK} implies that a vector field $K$
of the form (\ref{form}) is Killing on $B_f \times F$ if and
only if the following equations are satisfied:
\begin{equation} \label{2eqn}
\left\{
\begin{array}{rcl}
\psi^{\overline a} \sigma_{\overline a} + K_F(\theta)& = & 0 \\
 {\rm d} \phi^{\overline b} \otimes \hat{K}_{\overline b} +
\hat{C}_{\overline a} \otimes f^2 {\rm d} \psi^{\overline a} & = &
0,
\end{array}
\right.
\end{equation}
where $C_{\overline a}$ is a conformal-Killing vector field on $B$
with $\sigma_{\overline a} \in \mathcal C^\infty(B)$, i.e. ${\rm
L}_{C_{\overline a}}^B g_B= 2\sigma_{\overline a}g_B$ and
$\theta=\ln f$.

\medskip

Let $(F,g_F)$ be a Riemannian manifold of dimension $s$ admitting
at least one \textit{nonzero} Killing vector field on $(F,g_F)$
and $f \in C^\infty_{>0}(F)$. Thus, there exists a basis
$\{K_{\overline b} \in \mathfrak X(F) | \, \overline b = 1,\cdots,
\overline s \}$ for the set of Killing vector fields on $F$.

Let $I$ be an open interval of the form $I=(t_1,t_2)$ in $\mathbb R,$
where $-\infty \leq t_1 < t_2 \leq \infty$ furnished with the metric
$-{\rm d}t^2$. Recalling Remark \ref{rem:Lie deriv}, we
observe that the dimension of the set of conformal Killing vector
fields on $(I,-{\rm d}t^2)$ is infinite so that one cannot apply directly
the above procedure due to M. S\'{a}nchez before observing that the
form of conformal Killing vector fields on $(I,-{\rm d}t^2)$ is
explicit. Indeed, it is easy to prove that all the computations are
valid by considering the form of any conformal Killing vector field
on $(I,-{\rm d}t^2)$, namely $h \partial_t$ where $h \in C^\infty(I)$,
instead of the finite basis of conformal Killing vector fields in the
procedure of M. S\'{a}nchez .

\medskip

If we apply the latter adapted technique to the standard static
space-time $M=I _f\times F$ with the metric $g=f^2 g_I \oplus g_F$
where $g_I=-{\rm d}t^2$, then a vector field $K \in  \mathfrak
X(M)$ is a Killing vector field if and only if $K$
can be written in the form
\begin{equation}\label{eq:Killing structure}
   K= \psi h \partial_t +
\phi^{\overline b} K_{\overline b},
\end{equation}
where
$h$ and $\phi^{\overline b} \in C^\infty(I)$ for any $\overline b
\in\{1, \cdots, \overline m\}$ and $\psi \in C^\infty(F)$
satisfies the following version of System \eqref{2eqn}
\begin{equation} \label{2eqnst}
\left\{
\begin{array}{rcl}
 h^\prime \psi+ \phi^{\overline
b} K_{\overline b}(\ln f)& = & 0 \\
{\rm d} \phi^{\overline b} \otimes g_F(K_{\overline b},\cdot) +
g_I(h \partial_t, \cdot) \otimes f^2 {\rm d} \psi & = & 0.
\end{array}
\right.
\end{equation}

Thus, in order to study Killing vector fields on standard static
space-times we will concentrate our attention on the existence of
solutions for System \eqref{2eqnst}.

Since $\displaystyle{{\rm d} \phi^{\overline b} = (\phi^{\overline
b})^\prime {\rm d}t}$ with $\phi^{\overline b} \in \mathcal
C^\infty(I)$ and $g_I(h \partial_t, \cdot)=-h {\rm d}t$,
\eqref{2eqnst} is equivalent to
\begin{equation} \label{3eqnst}
\left\{
\begin{array}{rcl}
 h^\prime \psi+ \phi^{\overline
b} K_{\overline b}(\ln f)& = & 0 \\
(\phi^{\overline b})^\prime {\rm d}t \otimes g_F(K_{\overline
b},\cdot) & = & h {\rm d}t \otimes f^2 {\rm d} \psi ,
\end{array}
\right.
\end{equation}
and by raising indices in the second equation, it is also
equivalent to
\begin{equation} \label{3eqnst vector field}
\left\{
\begin{array}{lrcl}
(a)\qquad & h^\prime \psi+ \phi^{\overline
b} K_{\overline b}(\ln f)& = & 0 \\
(b)\qquad &(\phi^{\overline b})^\prime \partial_t \otimes
K_{\overline b} & = & h \partial_t \otimes f^2 \grad_F \psi .
\end{array}
\right.
\end{equation}

First of all, we will apply a separation of variables procedure to
the second equation in \eqref{3eqnst vector field}. Recall that
$\{K_{\overline b}\}_{1 \le {\overline b} \le {\overline m}}$ is a
basis of the Killing vector fields in $(F,g_F)$. Thus by simple
computations, it is possible to show that $(\ref{3eqnst vector
field}-b)$ implies the following equation for $(\phi^{\overline
b})^\prime$
\begin{equation}\label{eq:phi form 1}
\begin{array}{rcl}
   (\phi^{\overline b})^\prime(t) &=&
   [h(t)-h(t_0)]\gamma^{\overline b}+(\phi^{\overline b})^\prime (t_0)
   ,\\
   &=& \gamma^{\overline b} h(t) + \delta^{\overline
   b},\\
\end{array}
\end{equation}
where $\gamma^{\overline b}$ and $\delta^{\overline b}$
$(=-h(t_0)\gamma^{\overline b}+(\phi^{\overline b})^\prime (t_0),
\textrm{ for some fixed } t_0 \textrm{ in } I )$ are real
constants for any choice of ${\overline b}$ where $1 \le
{\overline b}\le {\overline m}$.

The solutions of the first order ordinary differential equation in
\eqref{eq:phi form 1} are given by
\begin{equation}\label{eq:phi form 2}
\phi^{\overline b}(t)=\gamma^{\overline b} \int_{t_0}^t h(s)ds +
\delta^{\overline b}t + \eta^{\overline b},
\end{equation}
where $\eta^{\overline b}$ are constant for all $\overline b$.

By introducing \eqref{eq:phi form 1} in $(\ref{3eqnst vector
field}-b)$, the last equation takes the following equivalent form:
\begin{equation}\label{eq:3eqnst vector field-b-equiv}
    h\partial_t \otimes[\gamma^{\overline b} K_{\overline b} - f^2 \grad_F
    \psi] = \partial_t \otimes [- \delta^{\overline b} K_{\overline b}].
\end{equation}
Thus, by recalling again that $\{K_{\overline b}\}_{1 \le
{\overline b} \le {\overline m}}$ is a basis of the Killing vector
fields in $(F,g_F)$, there results three different cases, namely.
\begin{description}
    \item[$h \equiv 0$]By \eqref{eq:phi form 2},
    \eqref{3eqnst vector field} takes the form
\begin{equation} \label{3eqnst vector field h=0}
\left\{
\begin{array}{lrcl}
(a)\qquad & (\delta^{\overline b}t + \eta^{\overline b}) K_{\overline b}(\ln f)& = & 0 \\
(b)\qquad &\delta^{\overline b} K_{\overline b} & = & 0 .
\end{array}
\right.
\end{equation}
Since $\{K_{\overline b}\}_{1 \le {\overline b} \le {\overline m}}$
is a basis, $(\ref{3eqnst vector field h=0}-b)$ implies that for any
$\overline b,$ we have $\delta^{\overline b}=0$.

Thus $K = \eta^{\overline b} K_{\overline b}$, so that $K$ is a
linear combination of the elements in the basis $\{K_{\overline
b}\}_{1 \le {\overline b} \le {\overline m}}$ and consequently, it
is a Killing vector field on $(F,g_F)$.

    \item[$h \equiv h_0 \neq 0 \,\, constant$] By  \eqref{eq:phi form 2},
    \eqref{3eqnst vector field} takes the form

\begin{equation} \label{3eqnst vector field h=h0}
\left\{
\begin{array}{lrcl}
(a)\qquad & \phi^{\overline b} K_{\overline b}(\ln f)& = & 0 \\
(b)\qquad &\displaystyle \frac{1}{h_0}\left(\gamma^{\overline b}
h_0 + \delta^{\overline b}\right) K_{\overline b} & = & f^2
\grad_F \psi,
\end{array}
\right.
\end{equation}
where
\begin{equation}\label{eq:phi form 2 h0}
   \phi^{\overline b}(t)=\underbrace{(\gamma^{\overline b}
   h_0 + \delta^{\overline b})}_{h_0\tau^{\overline b}}t +
    \underbrace{(\eta^{\overline b} - \gamma^{\overline b}
    h_0t_0)}_{\omega^{\overline b}}.
\end{equation}

\noindent Note that in particular, Equation $(\ref{3eqnst vector
field h=h0}-b)$ implies that $f^2 \grad_F \psi$ is Killing on
$(F,g_F)$ and gives the coefficients of $f^2 \grad_F \psi$ with
respect to the basis $\{K_{\overline b}\}_{1 \le {\overline b} \le
{\overline m}}$. On the other hand, differentiating $(\ref{3eqnst
vector field h=h0}-a)$ with respect to $t$ and then considering
$(\ref{3eqnst vector field h=h0}-b)$, we obtain
\begin{equation*}\label{}
\begin{array}{rcl}
   0&=&(\gamma^{\overline b} h_0 + \delta^{\overline b})
   K_{\overline b}(\ln f) \\
   &=& h_0(f^2 \grad_F \psi )(\ln f) \\
   &=&  h_0 f g_F(\grad_F \psi,\grad_F
f).\\
\end{array}
\end{equation*}

\noindent Furthermore, \eqref{eq:phi form 2 h0} and $(\ref{3eqnst
vector field h=h0}-a)$ imply that
\begin{equation*}\label{}
   \omega^{\overline b}K_{\overline b}(\ln f)=0.
\end{equation*}
Hence, we proved that \eqref{3eqnst vector field}
%
%
suffices to the following:
\begin{equation}
\label{3eqnst vector field h=h0 2} \left\{
\begin{array}{l}
f \in C^\infty_{>0}(F),\psi \in C^\infty(F);\\
 f^{2} \grad_F \psi \textrm{ is a Killing vector field on } (F,g_F)
 \textrm{ with }\\
\textrm{ coefficients } \{\tau_{\overline b}\}_{1 \le {\overline
b} \le {\overline m}} \textrm{ relative to the basis }
\{K_{\overline b}\}_{1 \le
{\overline b} \le {\overline m}};\\
(f^{2} \grad_F \psi)(\ln f)  =  0; \\
\forall {\overline b}: \phi^{\overline b}(t)=h_0 \tau^{\overline
b}t+\omega^{\overline b} \textrm{ with } \omega^{\overline b}\in
\mathbb{R} : \omega^{\overline b}
K_{\overline b}(\ln f)=0.\\
\end{array}
\right.
\end{equation}
It is easy to prove that \eqref{3eqnst vector field}
%
%
is also necessary for \eqref{3eqnst vector field h=h0 2}.

\noindent Hence we proved that $K= \psi h_0 \partial_t +
\phi^{\overline b} K_{\overline b}$ is Killing if and only if
\eqref{3eqnst vector field h=h0 2} is satisfied.

    \item[$h\,\, nonconstant$] The procedure for this case is a
    generalization of the previous, namely.

    \noindent Since $h$ is nonconstant, we can take $t_0$ such that
    $h(t_0)\neq 0$ and we will work on the subinterval $I_{t_0}$
    (see \textit{Notation \ref{notation: It0}}).
%
%

\noindent First of all, note that by applying the separation of
variables method in \eqref{eq:3eqnst vector field-b-equiv}, the
\underline{non}-constancy of $h$ implies that
\begin{equation}\label{eq:h no const 1}
   \left\{
\begin{array}{l}
\gamma^{\overline b} K_{\overline b} - f^2 \grad_F
    \psi=0\\
\forall {\overline b}: \delta^{\overline b}=0 .\\
\end{array}
   \right.
\end{equation}
Thus, by \eqref{eq:phi form 2},
\begin{equation}\label{eq:phi form 2 h no const}
 \phi^{\overline b}(t)=\gamma^{\overline b} \int_{t_0}^t h(s)ds
 + \eta^{\overline b}.
\end{equation}

\noindent On the other hand, by differentiating $(\ref{3eqnst
vector field}-a)$ with respect to $t$ and then by considering
\eqref{eq:h no const 1}, we obtain
\begin{equation*}\label{}
    h^{\prime \prime}\psi + h (f^2 \grad_F \psi)(\ln f)=0.
\end{equation*}

\noindent Besides, by considering \eqref{eq:h no const 1},
\eqref{eq:phi form 2 h no const} and again $(\ref{3eqnst vector
field}-a)$ there results
\begin{equation*}\label{}
   \left[h^\prime -  \frac{h^{\prime \prime}}{h}
   \int_{t_0}^t h(s)ds \right]\psi +
   \eta^{\overline b}K_{\overline b}(\ln f)=0.
\end{equation*}

\noindent Thus, we proved that \eqref{3eqnst vector field} is
sufficient to
\begin{equation}
\label{3eqnst vector field h no const} \left\{
\begin{array}{l}
f \in C^\infty_{>0}(F),\psi \in C^\infty(F);\\
 f^{2} \grad_F \psi \textrm{ is a Killing vector field on }
 (F,g_F)\textrm{ with }\\
\textrm{coefficients } \{\tau_{\overline b}\}_{1 \le {\overline b}
\le {\overline m}} \, \textrm{ relative to the basis }
\{K_{\overline b}\}_{1 \le {\overline b} \le {\overline m}};\\
h^{\prime \prime}\psi + h (f^2 \grad_F \psi)(\ln f)=0;\\
\forall {\overline b}: \phi^{\overline b}(t)=\tau^{\overline
b}\displaystyle \int_{t_0}^t h(s)ds +\omega^{\overline b} \textrm{
with } \omega^{\overline b}\in \mathbb{R}:\\
\displaystyle \left[h^\prime -  \frac{h^{\prime \prime}}{h}
\int_{t_0}^t h(s)ds \right]\psi + \omega^{\overline b}
K_{\overline b}(\ln f)=0  \textrm{ on } I_{t_0}.
\end{array}
\right.
\end{equation}
It is easy to prove that \eqref{3eqnst vector field} is also a
necessary condition for \eqref{3eqnst vector field h no const}, on
an interval where $h$ does not take the zero value.


\noindent It is not difficult to show that if $\displaystyle
-\frac{h^{\prime \prime}}{h}$ is nonconstant, then $\psi \equiv 0$.
Thus,
\begin{equation}\label{eq:h non constant}
K= \phi^{\overline b} K_{\overline b}\textrm{ with
}\phi^{\overline b}(t)=\omega^{\overline b}\textrm{ and
}\omega^{\overline b} \in \mathbb{R}: \omega^{\overline b}
K_{\overline b}(\ln f)=0.
\end{equation}

\noindent On the other hand, if $\displaystyle -\frac{h^{\prime
\prime}}{h}=\nu$ is constant, the second statement of \eqref{3eqnst
vector field h no const}, namely
\begin{equation}\label{eq:5.22 2}
   h^{\prime \prime}\psi + h (f^2 \grad_F \psi)(\ln f)=0,
\end{equation}
results equivalent to
\begin{equation} \label{eq:3eqnst eq-a 1}
\left\{
\begin{array}{rcl}
 -h^{\prime \prime} & = & \nu h \\
(f^2 \grad_F \psi)(\ln f) & = & \nu \psi.
\end{array}
\right.
\end{equation}
Thus,
\begin{equation} \label{eq:tuning 3}
h(t)= \left\{
\begin{array}{lcl}
a e^{\sqrt{-\nu}\,t} + b e^{-\sqrt{-\nu}\,t}\, \,
& \textrm{ if } & \nu \neq 0 \\
a t + b \, \,  &\textrm{ if }&  \nu = 0,\\
\end{array}
\right.
\end{equation}
where $a$ and $b$ are real constants.


\noindent

Hence, by \eqref{3eqnst vector field h no const}, \eqref{eq:h non
constant} and \eqref{eq:tuning 3}, the problem \eqref{3eqnst} is
equivalent to: (see Notation \ref{notation: It0} for the
definition of the interval $I_{t_0}$ relative to the function $h$)
%
\begin{equation}\label{eq:3eqnst 2 h no 0}
\left\{
\begin{array}{l}
(a)\left\{
\begin{array}{l}
f \in C^\infty_{>0}(F), \psi \equiv 0; \\
\phi^{\overline b}(t)=\omega^{\overline b} \textrm{ on } I_{t_0}
\textrm{ where }\omega^{\overline b} \in
\mathbb{R}: \omega^{\overline b} K_{\overline b}(\ln f)=0\\
\end{array}
\right.
%
\\
\textrm{or }\\
%
(b)\left\{
\begin{array}{l}
f \in C^\infty_{>0}(F),\psi \in C^\infty(F);\\
 f^{2} \grad_F \psi \textrm{ is a Killing vector field on }
(F,g_F)\\
\textrm{with } \textrm{coefficients } \{\tau_{\overline b}\}_{1
\le {\overline b} \le {\overline m}} \, \textrm{ relative to the } \\
\textrm{basis } \{K_{\overline b}\}_{1 \le
{\overline b} \le {\overline m}};\\
(f^{2} \grad_F \psi)(\ln f)  =  \nu \psi \textrm{ where }
\nu \, \textrm{is constant} ;\\
h \textrm{ is given in } \eqref{eq:tuning 3};\\
\forall {\overline b}: \phi^{\overline b}(t)=\tau^{\overline
b}\displaystyle \int_{t_0}^t h(s)ds +\omega^{\overline b} \textrm{
with }
\omega^{\overline b}\in \mathbb{R}:\\
\displaystyle h^\prime (t_0) \psi + \omega^{\overline
b}K_{\overline b}(\ln f)=0 \textrm{ on }
I_{t_0}.\\
\end{array}
\right.
\\
\end{array}
\right.
\end{equation}
\end{description}

\noindent It is easy to prove that if a set of functions $h$, $\psi$
and $\{\phi_{\overline b}\}_{1 \le \overline b \le \overline m}$,
satisfy \eqref{eq:3eqnst 2 h no 0} with a real interval $I$ instead
of $I_{t_0}$, then \eqref{3eqnst} is verified, that is, the vector
field given by \eqref{eq:Killing structure} on a standard static
space-time of the form $M=I _f\times F$ is Killing.

Hence, in the precedent discussion we proved the following result.

\begin{thm}\label{thm:Killing ssst}Let $(F,g_F)$ be a Riemannian
manifold, $f \in C^\infty_{>0}(F)$ and $\{K_{\overline b}\}_{1 \le
\overline b \le \overline m}$ a basis of Killing vector fields on
$(F,g_F)$. Let also $I$ be an open interval of the form
$I=(t_1,t_2)$ in $\mathbb R,$ where $-\infty \leq t_1 < t_2 \leq
\infty.$
Consider the standard static space-time $I _f\times F$ with the
metric $g=-f^2{\rm d}t^2 \oplus g_F$.

Then, any Killing vector field on $I _f\times F$ admits the
structure
\begin{equation}\label{eq:eq:Killing 2-s no 0}
K = \psi h \partial_t + \phi^{\overline b} K_{\overline b}
\end{equation}
where $h$ and $\phi^{\overline b} \in C^\infty(I)$ for any
$\overline b \in\{1, \cdots, \overline m\}$ and $\psi \in
C^\infty(F).$

Furthermore, assume that $K$ is a vector field on $I _f\times F$
with the structure as in \eqref{eq:eq:Killing 2-s no 0}. Then,
\begin{itemize}
\item [{\bf (i)}] if $h \equiv 0,$ then the vector field
$K=\phi^{\overline b} K_{\overline b}$ is Killing on the standard
static space-time $I _f\times F$ if and only if the functions
$\phi^{\overline b}$ are constant and $\phi^{\overline b}
K_{\overline b}(\ln f)=0$.

\item [{\bf (ii)}] if $h \equiv h_0 \neq 0$ is constant,
then the vector field $K= \psi h_0 \partial_t + \phi^{\overline b}
K_{\overline b}$ is Killing on the standard static space-time
$I_f \times F$ if and only if \eqref{3eqnst vector field h=h0 2}
is satisfied.

\item [{\bf (iii)}] if
$K$ is a Killing vector field on the standard static space-time
$I _f\times F$ with the nonconstant function $h$, then the set of
functions $h$, $\psi$ and $\{\phi_{\overline b}\}_{1 \le \overline
b \le \overline m}$ satisfy \eqref{eq:3eqnst 2 h no 0} for any
$t_0 \in I$ with $h(t_0)\neq 0.$

Conversely, if a set of functions $h$, $\psi$ and
$\{\phi_{\overline b}\}_{1 \le \overline b \le \overline m}$,
satisfy \eqref{eq:3eqnst 2 h no 0} with an arbitrary $t_0$ in $I$
and the entire interval $I$ (instead of $I_{t_0}$), then the
vector field $\tilde{K}$ on the standard static space-time
$I _f\times F$ associated to the set of functions as in
\eqref{eq:eq:Killing 2-s no 0} is Killing on $I _f\times F$.
\end{itemize}
\end{thm}

\bigskip

By completeness we consider in the following lemma the case where
the Riemannian manifold $(F,g_F)$ admits no non identically zero
Killing vector fields.

\begin{lem}\label{lem:0 Killing}
Let $(F,g_F)$ be a Riemannian manifold of dimension $s$ and $f \in
C^\infty_{>0}(F)$. Let also $I$ be an open interval of the form
$I=(t_1,t_2)$ in $\mathbb R,$ where $-\infty \leq t_1 < t_2 \leq
\infty$. Suppose that the only Killing vector field on $(F,g_F)$ is
the zero vector field. Then all the Killing vector fields on the
standard static space-time $I_f\times F$ are given by
$h_0 \partial_t$ where $h_0$ is a constant.
\end{lem}

\begin{proof}
Indeed, by Proposition \ref{prp:sanchez99} if $K$ is a Killing
vector field on $I _f\times F$, then $K=\psi h \partial_t$ where
$\psi \in C^\infty(F)$ and $h \in C^\infty (I)$. Then, by
similar arguments to those applied to the system \eqref{3eqnst
vector field}, a vector field of the latter form is Killing if and
only if the following equations are verified
\begin{equation*}
\left\{
\begin{array}{l}
(a) \qquad h^\prime \psi  =  0 \\
(b) \qquad h \partial_t \otimes f^2 \grad_F \psi  =  0.
\end{array}
\right.
\end{equation*}
As an immediate consequence, either $h$ and $\psi$ are constants
or $\psi \equiv 0$.
\end{proof}

\begin{rem}\label{rem:Killing non trivial}
If the Riemannian manifold $(F,g_F)$ admits a nonidentical zero
Killing vector field, then the family of Killing vector fields
obtained in \textit{Corollary \ref{cor:simple Killing}}
corresponds to the case of $\psi \equiv 1$ in \textit{Theorem
\ref{thm:Killing ssst} (iii)}. Thus, \eqref{eq:3eqnst 2 h no 0}
implies that $\nu =0$ and $\tau^{\overline b}=0$ for any
$\overline b$ and also $h(t)=at+b$ is affine and $\phi^{\overline
b}=\omega^{\overline b}$ is constant such that $\phi^{\overline b}
K_{\overline b}(\ln f)=\phi^{\overline b} K_{\overline b}(\ln
f)=-a$. The latter conditions agree with those in
\textit{Corollary \ref{cor:simple Killing}}.

In other words, if $\nu$ is nonzero, then the family of Killing
vector fields in \textit{Theorem \ref{thm:Killing ssst} (iii)} are
different form those in \textit{Corollary \ref{cor:simple Killing}},
they correspond to the so called non-trivial Killing vector fields
in \cite{mSK}.
\end{rem}

\begin{rem}\label{rem:subspace}
Let $f\in C^\infty_{>0}(F)$ be smooth. For any $\nu \in
\mathbb{R},$ we consider the problem given by
\begin{equation}
\label{eq:pb Killing 1} \left\{
\begin{array}{l}
f^{2} \grad_F \psi \textrm{ is a Killing vector field on }
(F,g_F);\\
(f^{2} \grad_F \psi)(\ln f)  =  \nu \psi ;\\
\psi \in C^\infty(F);\\
\end{array}
\right.
\end{equation}
and define    $\mathcal{K}_f^\nu= \{ \psi \in C^\infty(F): \psi
\textrm{ verifies } \eqref{eq:pb Killing 1} \}$. It is easy to
show that $\mathcal{K}_f^\nu$ is an $\mathbb{R}-$subspace of
$C^\infty(F)$. In particular, if $\psi \in \mathcal{K}_f^\nu$ then
\begin{equation}\label{eq:pb Killing 2}
(f^{2} \grad_F \lambda\,\psi)(\ln f)  =  \lambda\nu \,\psi, \,
\forall \lambda \in \mathbb{R}
\end{equation}
Consequently, if $\{\tau_{\overline b}\}_{1 \le {\overline b} \le
{\overline m}}$ is the set of coefficients of the Killing vector
field $f^{2} \grad_F \psi$ with respect to the basis $
\{K_{\overline b}\}_{1 \le {\overline b} \le {\overline m}}$ and
$\lambda \in \mathbb{R},$ then
\begin{equation}\label{eq:pb Killing 3}
-\lambda\nu \, \psi+\omega^{\overline b}K_{\overline b}(\ln f)= 0,
\end{equation}
where $\omega^{\overline b}=\lambda \tau^{\overline b},$ for any
$\overline b.$

Notice that, this is particularly useful in order to simplify the
hypothesis of \eqref{eq:3eqnst 2 h no 0} when $\nu \neq 0$.
\end{rem}

In order to analyze the existence of nontrivial solutions for the
problem \eqref{eq:pb Killing 1} (notice that this is relevant in
\eqref{eq:3eqnst 2 h no 0} and \eqref{3eqnst vector field h=h0
2}), we introduce the following notation.

\begin{notation} Let $(F,g_F)$ be a Riemannian manifold. Suppose
that $Z \in \mathfrak {X}(F)$ is a vector field on $(F,g_F)$ and
$\varphi \in C^\infty(F)$ is a smooth function on $F.$ Then define
a (0,2)-tensor on $F$ as follows:
\begin{equation}\label{eq:special Killing 2}
    B_Z^\varphi(\cdot,\cdot):=\textrm{d} \varphi (\cdot)
    \otimes g_F(Z,\cdot) + g_F(\cdot,Z) \otimes \textrm{d}
    \varphi (\cdot).
\end{equation}
\end{notation}

\begin{prp}\label{prp:special Killing}
Let $(F,g_F)$ be a Riemannian manifold, $f \in C^\infty_{>0}(F)$
and $\psi \in C^\infty(F)$. Then the vector field $f^{2} \grad_F
\psi$ is Killing on $(F,g_F)$ if and only if
\begin{equation}\label{eq:special Killing 1}
   {\rm H}_F^\psi + \frac{1}{f} B_{\grad_F \psi}^{f}=0,
\end{equation}
\end{prp}

\begin{proof} We begin by recalling two results. For all
smooth functions $\varphi \in C^\infty(F)$ and vector fields $X,Y
\in \mathfrak{X} (F),$ we have
\begin{equation*}\label{}
   {\rm H}_F^\varphi(X,Y) =
   g_F(\nabla_X^F \grad_F \varphi,Y)=
   \frac{1}{2} {\rm L}^F_{\grad_F \varphi} g_F(X,Y).
\end{equation*}

Moreover, for any vector field $Z \in \mathfrak X(F),$ the
following general formula can be stated.
\begin{equation*}\label{}
    {\rm L}_{\varphi Z}g_F(\cdot,\cdot)=\varphi {\rm L}_Z g_F(\cdot,\cdot) +
    \textrm{d} \varphi (\cdot) \otimes g_F(Z,\cdot) +
    g_F(\cdot,Z) \otimes \textrm{d} \varphi (\cdot).
\end{equation*}

By applying the latter formulas to the vector field $f^{2}
\nabla_F \psi,$ we obtain the following:
\begin{equation*}\label{}
   {\rm L}^F_{f^2 \grad_F \psi} g_F = f^2 {\rm L}_{\grad_F \psi}g_F +
   B_{\grad_F \psi}^{f^{2}} = 2 f^2 \left[{\rm H}^\psi_F + \frac{1}{f}
   B_{\grad_F \psi}^{f}\right].
\end{equation*}
Then one can conclude that $f^{2} \grad_F \psi$ is a Killing
vector field on $(F,g_F)$ if and only if \eqref{eq:special Killing
1} is satisfied.
\end{proof}

In order to study Killing vector fields on standard static
space-times, notice the central role of the problem given below,
i.e, (\ref{eq:3eqnst 4}) which appears throughout \textit{Theorem
\ref{thm:Killing ssst}}, \textit{Proposition \ref{prp:special
Killing}} and the identity stated as $fg_F(\grad_F \psi,\grad_F
f)=(f \grad_F \psi)(f),$
\begin{equation} \label{eq:3eqnst 4}
\left\{
\begin{array}{l}
f \in C^\infty_{>0}(F), \psi \in C^\infty(F);\\
\displaystyle {\rm H}_F^\psi + \frac{1}{f}
B_{\grad_F \psi}^{f} = 0 ;\\
fg_F(\grad_F \psi,\grad_F f) = \nu \psi \textrm{ where } \nu
\textrm{ is a constant}.
\end{array}
\right.
\end{equation}

\begin{rem}\label{rem:Lie algebra 1} By
\textit{Proposition \ref{prp:special Killing}}, if the dimension
of the Lie algebra of Killing vector fields of $(F,g_F)$ is zero,
then the system \eqref{eq:3eqnst 4} has only the trivial solution
given by a constant $\psi $ (this constant is not $0$ only if $\nu
=0$). This happens, for instance when $(F,g_F)$ is a compact
Riemannian manifold of negative-definite Ricci curvature without
boundary, indeed it is sufficient to apply the vanishing theorem
due to Bochner (see for instance \cite {Bochner46}, \cite[Theorem
1.84]{B} or \cite[Proposition 6.6 of Chapter III]{RG}).
\end{rem}

\begin{lem}\label{lem: laplace spectrum 1}
Let $(F,g_F)$ be a Riemannian manifold and $f \in
C^\infty_{>0}(F)$. If $(\nu,\psi)$ satisfies \eqref{eq:3eqnst 4},
then $\nu$ is an eigenvalue and $\psi$ is an associated
$\nu-$eigenfunction of the elliptic problem:
\begin{equation}\label{eq:weight Laplace-Beltrami 1}
    -\Delta_{g_F} \psi = \nu \frac{2}{f^2}
    \psi \, \textrm{ on } (F,g_F).
\end{equation}
\end{lem}

\begin{proof} First of all, note that
\begin{equation}\label{eq:bilinear 1}
    B_{\grad_F \psi}^f=\textrm{d} f \otimes \textrm{d} \psi +
    \textrm{d} \psi \otimes \textrm{d} f,
\end{equation}
then by taking the $g_F-$trace of the both sides, we have:
\begin{equation}\label{eq:bilinear 2}
    \trace B_{\grad_F \psi}^f= 2 g_F (\grad_F \psi,\grad_F f).
\end{equation}
Thus, by taking the $g_F-$trace of the first equation in
\eqref{eq:3eqnst 4} and then by applying the second equation of
\eqref{eq:3eqnst 4}, we obtain \eqref{eq:weight Laplace-Beltrami
1}. Hence, $\nu $ belongs to the spectrum of the weighted
eigenvalue problem \eqref{eq:weight Laplace-Beltrami 1}.
\end{proof}

\begin{rem}\label{rem:equivalent system}
Let $(F,g_F)$ be a Riemannian manifold.
\begin{description}
    \item[i]
Notice that similar arguments to those applied in Lemma \ref{lem:
laplace spectrum 1} allow us to prove that the system
\eqref{eq:3eqnst 4} is \textit{equivalent} to
\begin{equation} \label{eq:3eqnst 5}
\left\{
\begin{array}{l}
f \in C^\infty_{>0}(F), \psi \in C^\infty(F);\\
\displaystyle {\rm H}_F^\psi + \frac{1}{f}
B_{\grad_F \psi}^{f} = 0 ;\\
-\Delta_{g_F} \psi = \displaystyle \nu \frac{2}{f^2}
    \psi \textrm{ where } \nu
\textrm{ is a constant}.
\end{array}
\right.
\end{equation}
    \item[ii] Assuming \eqref{eq:3eqnst 5} (or equivalently \eqref{eq:3eqnst
    4}), if $p \in F$ is a critical point of $f$ or $\psi$, then $\nu = 0$
    or $\psi(p)=0$.
\end{description}
\end{rem}

\begin{prp}\label{prp:laplace spectrum 1} Let $(F,g_F)$ be a compact
Riemannian manifold and $f \in C^\infty_{>0}(F)$. Then $(\nu,\psi)$
satisfies \eqref{eq:3eqnst 4} if and only if $\nu =0$ and $\psi$ is
constant.
\end{prp}

\begin{proof} The necessity part is clear, so we will concentrate
our attention on the sufficiency part.

First of all, notice that by the second equation of
\eqref{eq:3eqnst 4}, if $p \in F$ is a critical point of $\psi$,
then $\nu \psi(p)=0$. Then, since $(F,g_F)$ is compact, there
exists a point $p_0 \in F$ such that $\psi(p_0)=\inf_F \psi$ and
consequently, $\nu \psi(p_0)=0$.

On the other hand, by applying \textit{Lemma \ref{lem: laplace
spectrum 1}}, one can conclude that $\nu$ is an eigenvalue and
$\psi$ is an associated $\nu-$eigenfunction of the elliptic
problem \eqref{eq:weight Laplace-Beltrami 1}. Besides,  since
$(F,g_F)$ is compact, it is well known that the eigenvalues of
\eqref{eq:weight Laplace-Beltrami 1} form a sequence in
$\mathbb{R}_{\ge 0}$ and the only eigenfunctions without changing
sign are the constants corresponding to the eigenvalue $0$.

Thus, if $\psi(p_0) \ge 0$, then $\nu =0$ and $\psi $ results a
nonnegative constant. Alternatively, if $\psi(p_0) < 0$, then $\nu
\psi(p_0)=0$, so $\nu =0.$ As a consequence of that, $\psi$ is a
negative constant.
\end{proof}

\textit{Theorem \ref{thm:Killing ssst}} and \textit{Proposition
\ref{prp:laplace spectrum 1}} provide the characterization of the
Killing vector fields on a standard static space-time when its Riemannian
part is compact (compare the result with \textit{Corollary
\ref{cor:simple Killing}}).

\begin{thm}\label{thm:killing compact fiber} Let $(F,g_F)$ be a
Riemannian manifold, $f \in C^\infty_{>0}(F)$
and $I$ be an open interval of the form $I=(t_1,t_2)$ in $\mathbb
R,$ where $-\infty \leq t_1 < t_2 \leq \infty$.
Consider the standard static space-time $I _f\times F$ with the
metric $g=-f^2{\rm d}t^2 \oplus g_F$.
If $(F,g_F)$ is compact then, the set of all Killing vector fields
on the standard static space-time $(M,g)$ is given by
$$\{a
\partial_t + \tilde{K}| \, a \in \mathbb{R}, \tilde{K} \textrm{ is a
Killing vector field on } (F,g_F) \textrm{ and } \tilde{K}(f)=0
\}.$$
\end{thm}

\begin{proof} If $(F,g_F)$ has only the zero
Killing vector field, the result is an easy consequence of Lemma
\ref{lem:0 Killing}.

Let us consider now the case there exists a basis $\{K_{\overline b}\}_{1
\le \overline b \le \overline m}$ for the space of Killing vector fields
on $(F,g_F)$. \textit{Theorem \ref{thm:Killing ssst}} and
\textit{Proposition \ref{prp:laplace spectrum 1}} imply that a
vector field $K$ on the standard static space-time $(M,g)$ is
Killing if and only if it admits the structure
\begin{equation}\label{eq:eq:Killing 2-s no 0 compact fiber}
K = \psi h \partial_t + \phi^{\overline b} K_{\overline b},
\end{equation}
where
\begin{enumerate}
\item $h(t)=a t+b$ with constants $a$ and $b$;
\item $\psi$ is constant;
\item $\phi^{\overline b}$ are constants satisfying
$a \psi + \phi^{\overline b} K_{\overline b}(\ln f)=0$.
\end{enumerate}
Since $(F,g_F)$ is compact, then $\inf_F \ln f$ is reached on a
point $p_0 \in F$. Set $\tilde{K}=\phi^{\overline b} K_{\overline
b}$. So $\tilde{K}(\ln f)|_{p_0}=0$ and by (3) $a=0$ or $\psi =0$.
Hence we proved that all the Killing vector fields on $I _f\times F$
are given by a Killing vector field on $(F,g_F)$ plus
eventually a real multiple of $\partial_t$. Note that
$\displaystyle \tilde{K}(\ln f)=\frac{1}{f} \tilde{K}(f)$, so by
(3) $\tilde{K}(f)=0$.
\end{proof}

\begin{rem} \label{rem:sharipov2007}
\textit{(Killing vector fields in the Einstein static universe)}
In \cite{Sharipov2007}, the author studied Killing vector fields
of a closed homogeneous and isotropic universe (for related
questions in quantum field theory and cosmology see \cite{dewitt
1975,dewitt 2003,Fulling1987, Landau Lifshits 1988}). Theorem 6.1
of \cite{Sharipov2007} corresponds to our Theorem \ref{thm:killing
compact fiber} for the spherical universe $\mathbb{R} \times
\mathbb{S}^3$ with the pseudo-metric $\displaystyle -( R^2 {\rm
d}t^2 -R^2 h_0)$, where the sphere $\mathbb{S}^3$ endowed with the
usual metric $h_0$ induced by the canonical Euclidean metric of
$\mathbb{R}^4$ and $R$ is a real constant (i.e., a stable
universe).
\end{rem}


As we have already mentioned in {\it Remark \ref{rem:Lie algebra
1}}, any Killing vector field of a compact Riemannian manifold of
negative-definite Ricci tensor is equal to zero. Thus, one can easily
state the following result.


\begin{cor} \label{cor:killing compact fiber of neggative definite
Ricci tensor} Let $M=I _f\times F$ be a standard static space-time
with the metric $g=-f^2{\rm d}t^2 \oplus g_F.$ Suppose that
$(F,g_F)$ is a compact Riemannian manifold of negative-definite
Ricci tensor. Then, any Killing vector field on the standard static
space-time $(M,g)$ is given by $a \partial_t$ where $a \in \mathbb
R.$
\end{cor}

\section{Non-rotating Killing vector fields }

In this section we apply Theorem \ref{thm:killing compact
fiber} of {\it Section 2} to the analysis of non-rotating
Killing vector fields on standard static space-times also
called \textit{static regular predictable space-times}
in \cite[p.325]{HE} (also see a recent article, i.e,
\cite{Sanchez-Senovilla2007} for a related question).

We first recall the definition of the $\curl$ operator on
semi-Riemannian manifolds as: if $V$ is a vector field on a
semi-Riemannian manifold $(M,g)$, then $\curl V$ is the
2-covariant tensor field defined by
\begin{equation}\label{eq:curl}
\curl V (X,Y):=g(\nabla_X V,Y)-g(\nabla_Y V,X),
\end{equation}
%
%
where $X,Y \in \mathfrak X(M)$ (see for instance
\cite{ON,Gutierrez-Olea2003}). Thus, it is easy to prove
\begin{eqnarray}\label{eq:curl 0}
\curl (\phi V) (X,Y)& = & X(\phi) g(V,Y)- Y(\phi)g( V,X) \\
& + & \phi \curl V (X,Y) \nonumber ,
\end{eqnarray}
where $X, Y$ are as above and $\phi \in C^\infty (M)$.

We will follow the next definitions (see
\cite{Gutierrez-Olea2003,Gutierrez-Olea2007}): A vector field $V$
on the semi-Riemannian manifold $(M,g)$ is said to be
%
%
\begin{description}
    \item[non-rotating]\footnote{In \cite {Gutierrez-Olea2003,Gutierrez-Olea2007}
    this condition is called irrotational.}
    if $\curl V (X,Y)=0$ for all $X,Y \in \mathfrak
    X(M)$.
    \item[orthogonally irrotational]\footnote{In \cite{ON,Sanchez-Senovilla2007}
    this condition is called irrotational.}
    if $\curl V (X,Y)=0$ for any $X,Y \in \mathfrak
    X(M)$ orthogonal to $V$. This condition is equivalent to
    ``$V$ has integrable orthogonal distribution".
\end{description}

Notice that being orthogonally irrotational is necessary for
being non-rotating. Furthermore \eqref{eq:curl 0} implies that
if $V$ is orthogonally irrotational, then so is $\phi V$ for any
$\phi \in C^\infty_{>0}(M)$. Indeed, since $\phi$ does not vanish,
$X$ is orthogonal to $\phi V$ if and only if it is orthogonal to
$V$.

\noindent However, if $V$ is non-rotating then, $\phi V$ is not
necessarily non-rotating for some $\phi \in C^\infty_{>0} (M)$
(see \eqref{eq:curl 0}).

\begin{prp}\label{prp:time irrotational} Let $\kappa$ be a Killing
vector field on the standard static space-time $(\mathbb{R} \times
F,g:=-f^2 dt^2 + g_F)$ such that $(\curl \kappa) (X,Y)=0$ for all
$X,Y$ orthogonal vector fields to $\partial_t$ defined on
$(\mathbb{R} \times F,g)$. If $(F,g_F)$ is compact and simply
connected, then $\kappa = a
\partial_t$ where $a$ is a real constant. In particular, $\kappa$
becomes time-like if $a \neq 0$.
\end{prp}

\begin{proof}
First of all we observe that it is easy to prove the following
identity.
\begin{equation}\label{eq:grad}
    -f^2 \grad t = \partial_t.
\end{equation}
Thus by using \eqref{eq:curl 0},
\begin{eqnarray} \label{eq:curl partial t}
   (\curl \partial_t)(X,Y) & = & \displaystyle
   \left(\curl f^2 (-\grad t)\right) (X,Y) \\
   & = & \displaystyle \frac{X(f^2)}{f^2} g(\partial_t,Y)-
   \frac{Y(f^2)}{f^2}g(\partial_t,X) \nonumber \\
   & - &f^2 (\underbrace{\curl \grad t}_{=0}) (X,Y) \nonumber,\\
   & = & \displaystyle \frac{X(f^2)}{f^2} g(\partial_t,Y)-
   \frac{Y(f^2)}{f^2}g(\partial_t,X) \nonumber.
\end{eqnarray}
Hence,
\begin{equation}\label{eq:curl partial t orth}
\begin{split}
&(\curl \partial_t) (X,Y)=0 \textrm{ for all } X,Y
\textrm{ orthogonal } \\
&\qquad\textrm{vector fields to } \partial_t \textrm{ on }
(\mathbb{R} \times F,g),
\end{split}
\end{equation}
i.e., $\partial_t$ is orthogonally irrotational. Up to this point,
we have not considered the compactness and the simply connectedness
of the Riemannian part.

Notice that letting $\kappa$ be a Killing vector field on
$(\mathbb{R} \times F,g)$ with $(F,g_F)$ compact, Theorem
\ref{thm:killing compact fiber} implies that $\kappa = a
\partial_t + K$ where $a$ is a constant and $K$ is a Killing
vector field on $(F,g_F)$ with $K(f)=0$. Hence and recalling that
$(\curl \kappa) (X,Y)=0$ for all $X,Y$ orthogonal vector fields to
$\partial_t$ on $(\mathbb{R} \times F,g)$, the $\curl$ on
$(\mathbb{R} \times F,g)$ is linear and \eqref{eq:curl partial t
orth}, we obtain that
\begin{equation} \label{eq:curl 1}
    \curl K (X,Y) = 0,
\end{equation}
for all $X,Y$ orthogonal vector fields to $\partial_t$ on
$(\mathbb{R} \times F,g)$. It is clear that the lifts of the
elements in $\mathfrak X(F)$ verify the latter condition and as
consequence \eqref{eq:curl 1}.

On the other hand, by the well known relations between the
Levi-Civita connections on $(F,g_F)$ and on the warped product
$(\mathbb{R} \times F,$ $-f^2 dt^2 + g_F)$ (see for instance
\cite{ON} p. 206 Proposition 35) and by the definition of the
$\curl$, there results $\curl_F K = 0$ ($\curl_F$ denotes the
$\curl$ on the Riemannian manifold $(F,g_F)$), i.e. $K$ is
non-rotating on $(F,g_F)$. Furthermore, when $(F,g_F)$ is simply
connected, then $K$ is a gradient (see
\cite{Bishop-Goldberg1980}), i.e. there exists $\psi \in C^\infty
(F)$ such that $\grad \psi = K$.


Since $K$ is a Killing vector field on the Riemannian manifold
$(F,g_F)$ and is the gradient of a function $\psi \in C^\infty
(F)$, then Proposition \ref{prp:special Killing} (with $f
\equiv 1$) implies that the Hessian $H_F^\psi \equiv 0$ on $F$.
Besides that if $F$ is a compact Riemannian manifold without
boundary, then by taking the $g_F-$trace of $H_F^\psi$ and applying
the maximum principle one can deduce that $\psi$ is a constant and
consequently, $K \equiv 0$.

Thus we have established that $\kappa = a \partial_t$, where $a$ is
constant.
\end{proof}

Notice that \eqref{eq:curl partial t} implies that $\partial_t$ is
non-rotating on the standard static space-time of the form
$(\mathbb{R} \times F,-f^2 dt^2 + g_F)$ if $f$ is a positive
constant. Hence, we can state and  prove the following lemma.

\begin{lem}\label{lem:partial t irrot}$\partial_t$ is non-rotating
on the standard static space-time of the form
$(\mathbb{R} \times F,g:=-f^2 dt^2 + g_F)$ if and only if $f$ is a
positive constant.
\end{lem}
%
%
\noindent\emph{Proof 1.} First of all we recall
$g(\partial_t,\partial_t)=-f^2$. Then, by the hypothesis and the
properties that characterized the Levi-Civita connection, for any
vector field $Z$ on $(\mathbb{R} \times F,g:=-f^2 dt^2 + g_F)$
the following holds
\begin{equation}\label{eq:partial t - 1}
\begin{split}
    0&= \curl \partial_t (\partial_t,Z) = g(\nabla _{\partial_t}
    \partial_t,Z) - \underbrace{g(\nabla _{Z}
    \partial_t,\partial_t)}_{\frac{1}{2} Z
    g(\partial_t,\partial_t)}\\
    &=g(\nabla _{\partial_t}
    \partial_t,Z) + \displaystyle \frac{1}{2} Z f^2.
\end{split}
\end{equation}
So, by the definition of the gradient operator denoted by $\grad$,
\begin{equation}\label{eq:partial t - 2}
-\nabla _{\partial_t} \partial_t = \displaystyle \frac{1}{2} \grad
f^2.
\end{equation}
Now, by applying \cite[p. 206 - Proposition 35]{ON} or
\cite[Theorem 3.4]{dobarro-unal04} as above,
\begin{equation}\label{eq:partial t - 3}
-\nabla _{\partial_t} \partial_t = - \underbrace{\nabla^{(-dt^2)}
_{\partial_t}
\partial_t}_{0} + f \underbrace{(-dt^2)(\partial_t,\partial_t)}_{-1} \grad
f =-\displaystyle \frac{1}{2} \grad f^2.
\end{equation}
The latter and the expression \eqref{eq:partial t - 2} imply
$\grad f^2 = 0$, more explicitly $f$ is a positive constant.
%
\hfill $\square$

\medskip

%
%
\noindent\emph{Proof 2.} An alternative and simpler proof follows
from \eqref{eq:curl partial t}. Indeed, an immediate consequence
of the latter is
\begin{equation}\label{eq:partial t - 4}
    (\curl \partial_t)(\partial_t,Z) = Z (f^2),
\end{equation}
for any vector field $Z$ on $(\mathbb{R} \times F,g:=-f^2 dt^2 +
g_F)$. Thus, if $\partial_t $ is non-rotating, then $f$ is a
positive constant.
%
\hfill $\square$
%


\begin{cor}
\label{cor:irrotational}Let $(F,g_F)$ be a compact and simply
connected Riemannian manifold. If $f \in C^\infty_{>0} (M)$ is
nonconstant, then there is no nontrivial (non-identically zero)
non-rotating Killing vector field on the standard static
space-time $(\mathbb{R} \times F,g:=-f^2 dt^2 + g_F)$.
\end{cor}

\begin{proof} In order to apply Proposition \ref{prp:time irrotational},
it is sufficient to note that if $\kappa$ is a non-rotating
Killing vector field on $(\mathbb{R} \times F,g:=-f^2 dt^2 + g_F)$,
then $(\curl \kappa) (X,Y)=0$ for all $X,Y$ orthogonal vector
fields to $\partial_t$ on $(\mathbb{R} \times F,g)$. Thus, $\curl
\kappa = a \curl \partial_t = 0$. So, by Lemma \ref{lem:partial t
irrot}, $\kappa=0$ or $f$ is a positive constant.
\end{proof}

\begin{rem} Notice that in Proposition \ref{prp:time irrotational}
the involved vector fields necessarily turn out to be causal
(i.e., non-space-like). The conclusion would be wrong if we
eliminated the simply connectedness. For example, consider the
vector field $a \partial_t + \partial_\theta$ where $a<1$ on
$(\mathbb{R} \times \mathbb{S}^1,-{\rm d}t^2 + {\rm d}\theta^2)$.
This is a non-rotating Killing vector field and yet not time-like
due to $a<1$, so its pseudo-norm is $-a^2+1>0$.
\end{rem}

\section{Conclusions}

It is very well known that a space-time possesses a symmetry if it
admits a Killing vector fields. Thus existence and
characterization problems of Killing vector fields are extremely
important in the geometry of space-times. Therefore, we considered
Killing vector fields of standard static space-times and roughly
proved that if a vector field $K$ on the space-time is Killing,
then $K$ has to be of the form:
$$ K= \psi h \partial_t + \phi^{\overline b} K_{\overline b}
,$$ where $\{K_{\overline b}\}_{1 \le \overline b \le \overline
m}$ is a basis of Killing vector fields on $(F,g_F)$ and $h,
\phi^{\overline b} \in \mathcal C^\infty(I)$ and $\psi \in
\mathcal C^\infty(F)$ satisfy the system
\begin{equation*}
\left\{
\begin{array}{rcl}
 h^\prime \psi+ \phi^{\overline
b} K_{\overline b}(\ln f)& = & 0, \\
{\rm d} \phi^{\overline b} \otimes g_F(K_{\overline b},\cdot) +
g_I(h \partial_t, \cdot) \otimes f^2 {\rm d} \psi & = & 0.
\end{array}
\right.
\end{equation*}
In {\it Theorem \ref{thm:Killing ssst}} we have characterized the
solutions of the latter system. We remarked the centrality of the
problem \eqref{eq:pb Killing 1} or equivalently \eqref{eq:3eqnst
4} and \eqref{eq:3eqnst 5}.

As a consequence of {\it Theorem \ref{thm:Killing ssst}} and the
well known results about the eigenvalues and eigenfunctions of a
positively weighted elliptic problem on a compact Riemannian
manifold without boundary, we also provided a characterization of
the Killing vector fields on a standard static space-time with
compact Riemannian part in {\it Theorem \ref{thm:killing compact
fiber}}. Note that by combining this theorem with the vanishing
results of Bochner (see \textit{Remark \ref{rem:Lie algebra 1}}),
we obtain that in a standard static space-time with compact
Riemannian part of negative Ricci curvature without boundary, the
unique Killing vector fields are time-like of the form $c
\partial_t$ where $c \in \mathbb{R}$ is constant.

Furthermore, in Corollary \ref{cor:irrotational} we show that there
is no non trivial non-rotating Killing vector field on a standard
static space-time where the so called natural space is compact and
simply connected.

\appendix{}
\section{}

Suppose that $\mathscr{M}$ is a module over a ring and
$\mathscr{W} \subseteq \mathscr{M}$. If $V \in \mathscr{M}$, then
we will use the following notation $V + \mathscr{W} = \{V + W : W
\in \mathscr{W} \}$.

Let $(F,g_F)$ be a Riemannian manifold of dimension $s$. Consider
the Lie algebra of Killing vector fields on
$(F,g_F)$ denoted by $\mathscr{K}$.
Given $\varphi, \psi \in C^\infty (F)$ we denote
\begin{equation}\label{}
    \mathscr{K}_\varphi ^\psi = \{K \in \mathscr{K}: K(\varphi)=\psi
    \}.
\end{equation}
$\mathscr{K}_\varphi ^0$ is an $\mathbb{R}-$subspace of
$\mathscr{K}$ (it is also a sub-Lie algebra of $\mathscr{K}$).

\noindent By linear algebra arguments, it is clear that
\begin{equation}\label{eq:K-lin-man}
    \mathscr{K}_\varphi ^\psi = \widehat{K} + \mathscr{K}_\varphi ^0,
\end{equation}
where $\widehat{K} \in \mathscr{K}_\varphi ^\psi$ is fixed.

\bigskip

Let us assume the hypothesis of Theorem \ref{thm:Killing ssst} where
$(\nu, \psi)$ is a solution of \eqref{eq:pb Killing 1}. Then, by
Theorem \ref{thm:Killing ssst} and \eqref{eq:K-lin-man}, the set
of Killing vector fields on $I _f\times F$ is given by
\begin{equation}\label{eq:Killing vf app}
\psi h \partial_t + \int_{t_0}^t h(s) ds f^2\grad \psi +
\widehat{K} + \mathscr{K}_{\ln f}^0,
\end{equation}
where $\widehat{K} \in \mathscr{K}_{\ln f} ^{-h'(t_0)\psi}$ is
fixed.

\noindent If $\nu \neq 0$, then $\widehat{K}$ may be taken as
$\displaystyle -\frac{h'(t_0)}{\nu} f^2 \grad \psi$. Indeed, this
is an immediate consequence of $(f^2 \grad_F \psi)(\ln f)=\nu
\psi$ in \eqref{eq:3eqnst 2 h no 0}.

\noindent Notice that the description of the Killing vector fields
on standard static space-times given by \eqref{eq:Killing vf app}
shows up again in the centrality of the problem \eqref{eq:pb
Killing 1} (and naturally of the equivalent problems
\eqref{eq:3eqnst 4} and \eqref{eq:3eqnst 5}) in the study of this
question.

\begin{ex}(\textit{Killing vector fields in the Minkowski space-time})
Let $(F,g_F)$ $= (\mathbb{R}^s,g_0)$ where $g_0$ is the canonical
metric and $f \equiv 1$ be. Thus, it is easy to show that the
solutions of \eqref{eq:3eqnst 5} are $(\nu,\psi)$ where $\nu = 0$
or $\psi \equiv 0$. Furthermore if $\nu = 0$, then $\psi(x)=c^i
x_i+d$ where $\forall i: 1\le i \le s, c^i \in \mathbb{R}$ and $d
\in \mathbb{R}$. Recall that for $\nu =0$, $h(t)=at+b$ where $a,b
\in \mathbb{R}$.

\noindent On the other hand, the condition $\widehat{K} \in
\mathscr{K}_{\ln f} ^{-h'(t_0)\psi}$ is now equivalent to $h'(t_0)
(c^i x_i+d) \equiv 0$.

\noindent Hence, all the Killing vector fields of the Minkowski
space-time are
\begin{equation*}
    (c^i x_i+d)(at+b) \partial_t +
\int_{t_0}^t (as+b) ds  \, c^i \partial_i + \mathscr{K}
   ,
\end{equation*}
where $a,b,c_i,d \in \mathbb{R}$ satisfy $a(c^i x_i+d) \equiv 0 $.
Precisely, these are
\begin{equation*}
    (c^i x_i+d) \partial_t + (t-t_0)
\, c^i \partial_i + \mathscr{K}
\end{equation*}
or equivalently (taking $t_0=0$)
\begin{equation*}
    c^i\underbrace{( x_i \partial_t + t \partial_i )}_{\textrm{Lorentz boosts}} +
    d \partial_t + \mathscr{K}
   ,
\end{equation*}
where $c_i,d \in \mathbb{R}$. Thus the dimension of the Lie
algebra of the Killing vector fields of the Minkowski space-time
is $\displaystyle s + 1 + s(s+1)/2 = (s+1)(s+2)/2$.


\end{ex}



\begin{thebibliography}{}

\bibitem{AD1} Allison D.E., \textit{Energy conditions in
standard static space-times}, General Relativity and
Gravitation, {\bf 20}, No. 2, (1988), 115-122.

\bibitem{AD} Allison D.E., \textit{Geodesic completeness
in static spacetimes}, Geometriae Dedicata  {\bf 26},
(1988), 85-97.

\bibitem{ADt} Allison D. E. \textit{Lorentzian warped
products and static space-times}, Ph.D. Thesis, University of
Missouri-Columbia, 1985.

\bibitem{GES} Allison D. E. and \"Unal B., \textit{Geodesic
Structure of Standard Static Space-times}, Journal of
Geometry and Physics, {\bf 46} (2), (2003), 193-200.

\bibitem{BEE} Beem J. K., Ehrlich P. E. and Easley K. L.,
\textit{Global Lorentzian Geometry}, (2nd Ed.), Marcel
Dekker, New York, 1996.

\bibitem{B} Besse, A. L., \textit{Einstein Manifolds},
Springer-Verlag, Heidelberg, 1987.

\bibitem{Bishop-Goldberg1980}Bishop, R. L., and S.I.
Goldberg S.I., \textit{Tensor Analysis on Manifolds}, Dover, New
York, 1980.

\bibitem{B-O}
B. Bishop and B. O'Neill, \textit{Manifolds of negative
curvature}, Trans. Am. Math. Soc., \textbf{145} (1969), 1-49.

\bibitem{Bochner46} Bochner S.,
\textit{Vector fields and Ricci curvature}, Bull. Am. Math. Sc.,
{\bf 52}, (1946), 776-797.

\bibitem{Carot-da Costa 93} Carot J. and da Costa J.,
\textit{On the geometry of warped spacetimes}, Class. Quantum
Grav., {\bf 10}, (1993), 461-482.

\bibitem{dewitt 1975} DeWitt B. S., Quantum field theory in curved
space-time, PHYSICS REPORTS (Section C of Physics Letters) 19, no.
6 (1975) 295-357.

\bibitem{dewitt 2003} DeWitt B. S.,
\textit{The global approach to quantum field theory}, Oxford
Science Publications, 2003.


\bibitem{SSS} Dobarro F. and \"Unal B., \textit{Special
standard static spacetimes}, Nonlinear Analysis: Theory, Methods
\& Applications, {\bf 59} (5), (2004), 759-770.

\bibitem{dobarro-unal04} Dobarro F. and \"Unal B.,
\textit{Curvature in Special Base Conformal Warped Products},
arXiv:math/0412436.

\bibitem{dobarro-unal06} Dobarro F. and \"Unal B.,
\textit{On Warped Product Space-Times: Conformal Hyperbolicity and
Killing Vector Fields}, arXiv:math/0607113.


\bibitem{Fulling1987} Fulling S. A.,
\textit{Aspects of quantum field theory in curved space-time},
Cambridge University Press, UK, 1987.

\bibitem{Gutierrez-Olea2003}
Guti\'errez M. and Olea B., \textit{Splitting Theorems in Presence
of an Irrotational Vector Field}, arXiv:math/0306343.

\bibitem{Gutierrez-Olea2007}
Guti\'errez M. and Olea B., \textit{Global Decomposition of a
Lorentzian Manifold as a Generalized Robertson-Walker Space},
arXiv:math/0701067.

\bibitem{HE} Hawking S. W. and Ellis G. F. R., \textit{The Large
Scale Structure of Space-time}, Cambridge University Press,
UK, 1973.

\bibitem{KO} Kobayashi O. and Obata, M., \textit{Certain
mathematical problems on static models in general
relativity}, Proceedings of the 1980 Beijing Symposium on
Differential Geometry and Differential Equations, Vol. 3,
(eds. S. S. Chern and W. Wen-ts\"un), 1980, 1333-1344.


\bibitem{Landau Lifshits 1988} Landau L. D. and Lifshits E. M.,
\textit{Field Theory}, Vol. II of \textit{Theoretical Physics},
Nauka publishers, Moscow, 1988.

\bibitem{ON} O'Neill, B., \textit{Semi-Riemannian Geometry
With Applications to Relativity}, Academic Press, New York,
1983.

\bibitem{RASM} Romero, A. and S\'{a}nchez M.,
\textit{On the completeness of certain families of
semi-Riemannian manifolds}, Geometriae Dedicata {\bf 53},
(1994), 103-117.

\bibitem{RG} Sakai T., \textit{Riemannian Geometry},
Translation of Mathematical Monographs {\bf 149}, AMS, RI, 1996.

\bibitem{mS1997} S\'{a}nchez, M.,
\textit{Lorentzian manifolds admitting a Killing vector field},
Nonlinear Analysis {\bf 30}, (1997), 643-654.

\bibitem{mSK} S\'{a}nchez, M., \textit{On the geometry of
generalized Robertson-Walker spacetimes: curvature and Killing
fields}, J. Geom. Phys. {\bf 31}, (1999), 1-15.

\bibitem{mS2} S\'{a}nchez, M.,
\textit{Geodesic connectedness of semi-Riemannian manifolds},
Nonlinear Analysis {\bf 47}, (2001), 3085-3102.

\bibitem{MS05} S\'{a}nchez, M., \textit{On the geometry of
static spacetimes}, Nonlinear Analysis: Theory, Methods \&
Applications {\bf 63}, (2005), e455-e463.


\bibitem{Sanchez-Senovilla2007} S\'{a}nchez, M. and Senovilla J. M. M.,
\textit{A note on the uniqueness of global static decompositions},
arXiv:0709.0305v1[gr-qc].

\bibitem{Sharipov2007} Sharipov R. A., \textit{On Killing vector fields of a
homogeneous and isotropic universe in closed model},
arXiv:0708.2508v1[math.DG].


\bibitem{PHD} \"Unal B., \textit{Doubly Warped Products},
Ph. D. Thesis, University of Missouri-Columbia, 2000.

\bibitem{DWP} \"Unal B. \textit{Doubly Warped Products},
Differential Geometry and Its Applications, {\bf 15} (3), (2001),
253-263.


\end{thebibliography}
\end{document}